\documentclass[12pt]{article}\textwidth 160mm\textheight 235mm
\usepackage{amsmath}
\usepackage{amsthm}
\usepackage{amsfonts}
\usepackage{multirow}
\usepackage[table,xcdraw]{xcolor}
\usepackage{subcaption}
\usepackage{enumerate}
\usepackage{float}
\usepackage{algorithm}
\usepackage{algpseudocode}
\usepackage{hyperref}
\usepackage{tabularx}
\usepackage{amssymb}
\usepackage{graphicx}
\usepackage{tcolorbox}
\usepackage{booktabs}
\usepackage[table,xcdraw]{xcolor}
\usepackage{multirow}
\usepackage{rotating,tabularx}

\usepackage[left=2cm,right=2cm,top=2cm,bottom=2cm]{geometry}
\newcolumntype{L}[1]{>{\raggedright\let\newline\\\arraybackslash\hspace{0pt}}m{#1}}
\newcolumntype{C}[1]{>{\centering\let\newline\\\arraybackslash\hspace{0pt}}m{#1}}
\newcolumntype{R}[1]{>{\raggedleft\let\newline\\\arraybackslash\hspace{0pt}}m{#1}}

\newtheorem{Theorem}{Theorem}

\newtheorem{Remark}[Theorem]{Remark}
\newtheorem{Lemma}[Theorem]{Lemma}
\newtheorem{Corollary}[Theorem]{Corollary}

\newtheorem{Example}[Theorem]{Example}
\newtheorem{example}[Theorem]{Example}

\usepackage{hyperref}
\hypersetup{colorlinks=true,urlcolor=blue}
\expandafter\let\expandafter\oldproof\csname\string\proof\endcsname
\let\oldendproof\endproof
\renewenvironment{proof}[1][\proofname]{
\oldproof[\ttfamily\scshape \bf #1.]
}{\oldendproof}

\def\conv{{\rm conv}\,}

\def\ra{\rangle}
\def\la{\langle}

\def\epsilon{\varepsilon}

\def \R{{\rm I\!R}}

\title{\bf Lipschitz Modulus of Convex Functions via Function Values}

\author{Pham Duy Khanh\footnote{Department of Mathematics, Ho Chi Minh City University of Education, Ho Chi Minh City, Vietnam. E-mail: khanhpd@hcmue.edu.vn.} \quad  Vu Vinh Huy Khoa\footnote{Department of Mathematics, Wayne State University, Detroit, Michigan, USA. E-mail: khoavu@wayne.edu.}  \quad  Vo Thanh Phat\footnote{Department of Mathematics  and Statistics, University of North Dakota, Grand Forks, North Dakota, USA. E-mail: vtphat1996@gmail.com.} \quad Le Duc Viet\footnote{Department of Mathematics, Wayne State University, Detroit, Michigan, USA. E-mail: vietle@wayne.edu.}}
\begin{document}
\maketitle

\noindent
{\small{\bf Abstract}. In this note, we establish the Lipschitz continuity of finite-dimensional globally convex functions on all given balls and global Lipschitz continuity for eligible functions of that type. The Lipschitz constants in both situations draw information solely from function values, and the global Lipschitz modulus is found when it exists. Some examples of classes of globally Lipschitz continuous convex functions beside the norms are also provided along with their global Lipschitz modulus. \\[1ex]
{\bf Mathematics Subject Classification (2020)}.  } 26A16, 52A41.

\section{Introduction}

Convexity is well-known as an attempt to generalize optimization theory beyond the classical smooth setting, while Lipschitz continuity plays a huge role in assessing the stability of optimization solutions and the convergence of numerical algorithms. The local Lipschitz continuity of convex functions is a classical property that has been well-established in Convex Analysis. For instance, in \cite{roberts}, the authors provided a proof for local Lipschitz continuity of convex functions on open sets, which establishes for every point the existence of a neighborhood in which the function is Lipschitz continuous. However, determining the Lipschitz constant for such continuity is not a trivial feat, and in many cases requires the evaluation of the derivative (or generalized derivatives). Here, we present the proof for a stronger claim specialized for finite-dimensional real-valued globally convex functions, stating that the function is Lipschitz continuous on all given spherical balls. Furthermore, we also establish the Lipschitz constant which draws information solely from the function values at finitely many points on the boundary of an associated convex set. The idea of controlling the information on the boundary is not new - for instance, it has been utilized in \cite{khanh} where the authors used the subdifferential information of the given function. However, using only the function values is a novel approach to this matter. As a subsequent effort, we also established the Lipschitz modulus for globally Lipschitz continuous convex functions and provided a characterization for such functions in finite-dimensional spaces. Namely, a convex function is globally Lipschitz continuous if and only if it is asymptotically majorized by a norm as the norm of the variable is sufficiently large. Some further discussions and illustrating examples provide some insight on globally Lipschitz continuous convex functions beyond the norms.
\section{Main results}
Let $\R^n$ be the standard Euclidean space of dimension $n$ with the Euclidean norm signified by $\|.\|$. The corresponding open ball with center $x_0 \in \R^n$ and radius $r > 0$ is denoted by $B(x_0,r)$. A function $f: \R^n \to \R$ is \textit{convex} if for all $x, y \in \R^n$, for all $\lambda \in [0,1]$,
$$f(\lambda x + (1 - \lambda)y) \le \lambda f(x) + (1 - \lambda) f(y). $$
A function $f: \R^n \to \R$ is Lipschitz continuous on a set $S \subset \R^n$ if there exists $L \ge 0$, which we call the {\em Lipschitz constant}, such that
\begin{equation}\label{ep:Lip}
\forall x, y \in S, \quad |f(x) - f(y)| \le L\|x - y\|.
\end{equation}
The infimum of all $L$ which satisfies \eqref{ep:Lip} is called the \textit{Lipschitz modulus} of $f$ on $S$. If the set $S$ in question is $\R^n$, we say that $f$ is \textit{globally Lipschitz continuous}.
For a convex set $S$, let $\mathrm{extr} (S)$ signify the set of \textit{extreme points} of $S$. Let $\conv (S)$ denote the \textit{convex hull} of a (not necessarily convex) set $S$. As a quick observation, the supremum of a convex function on any set $S$ is precisely the supremum of that function on the (generally bigger) convex set $\conv (S)$. \medskip

The following technical lemma provides a parametrical Lipschitz constant of a convex function on a specified ball that extracts information exclusively from the function values on a polytope which contains a magnified version of that ball. 
\begin{Lemma}\label{lem:lipschitz} Let $f: \R^n \to \R$ be convex. Pick any $x_0 \in \R^n$, radius $r > 0$ and two parameters $\lambda\in (0,1)$, $\alpha > \max \left\{ 1, \dfrac{\lambda}{1-\lambda}\right\}$. Let $S$ be a finite set of points such that $B(x_0,\alpha r) \subset \conv (S)$. Then, $f$ is Lipschitz continuous on $B(x_0,r)$ with constant
\begin{equation}\label{eq:lipschitzS}
    L(\lambda,\alpha,r) = \dfrac{1}{r\lambda (\alpha - 1)}\cdot\left[\max \{ f(z) \mid z \in S \} - f(x_0)\right].
\end{equation}
\end{Lemma}
\begin{proof} First, we establish the upper bound and the lower bound of $f$ on suitable balls using the function values of $f$ on $S$. Since $f$ is convex,
    \begin{equation}\label{eq:upper}
        \sup\limits_{z \in B(x_0,\alpha r)} f(z) \le \sup\limits_{z \in \conv (S)} f(z) = \sup\limits_{z \in S} f(z) = \max \{ f(z) \mid z \in S \} =: M. 
    \end{equation}
    Pick any $z \in B\left(x_0,\alpha r\cdot \dfrac{1-\lambda}{\lambda}\right)$, and define $w \in \R^n$ by $w:= \dfrac{1}{1 - \lambda}x_0 - \dfrac{\lambda}{1 - \lambda}z. $
    Then, it is easily seen that $x_0 = \lambda z + (1-\lambda)w$ and
    $$\|w - x_0\| = \left\|\dfrac{\lambda}{1 - \lambda} x_0 - \dfrac{\lambda}{1 - \lambda}z \right\| = \dfrac{\lambda}{1 - \lambda}\|x_0 - z\| \le \dfrac{\lambda}{1 - \lambda}\cdot \alpha r\cdot \dfrac{1-\lambda}{\lambda} = \alpha r, $$
    which entails $w \in B(x_0,\alpha r)$. Since $f$ is convex, $f(x_0) \le \lambda f(z) + (1 - \lambda) f(w)$, which means $$f(z) \ge \dfrac{1}{\lambda}f(x_0) - \dfrac{1-\lambda}{\lambda} f(w) \ge \dfrac{1}{\lambda}f(x_0) - \dfrac{1-\lambda}{\lambda}\cdot M. $$
    Hence, 
    \begin{equation}\label{eq:lower}
      \inf\limits_{z \in B\left(x_0,\alpha r\cdot \frac{1-\lambda}{\lambda}\right)} f(z) \ge \dfrac{1}{\lambda}f(x_0) - \dfrac{1-\lambda}{\lambda}\cdot M.
    \end{equation}
Now, pick distinct elements $x, y \in B(x_0,r)$. Construct the univariate function $g: \R \to \R$ as follows:
    $$\forall t \in \R, \quad g(t) := f(tx + (1-t)y), $$
    then $g$ is convex on $\R$ by the convexity of $f$ on $\R^n$. By the secant inequality for convex functions of one variable, we have
    \begin{equation}\label{eq:secant}
       \dfrac{g(0) - g \left( \dfrac{-r(\alpha-1)}{\|y - x\|}\right)}{\dfrac{r(\alpha-1)}{\|y-x\|}} \le \dfrac{g(1) - g(0)}{1} \le \dfrac{g \left( \dfrac{r(\alpha-1)}{\|y - x\|} + 1\right) - g(1)}{\dfrac{r(\alpha-1)}{\|y-x\|}}.
    \end{equation}
    Set $z_1 := y + \dfrac{r(\alpha-1)}{\|y - x\|}(y - x)$, $z_2 := x + \dfrac{r(\alpha-1)}{\|y - x\|}(x - y)$, then 
    $$g \left( \dfrac{-r(\alpha-1)}{\|y - x\|}\right) = f(z_1), \quad g \left( \dfrac{r(\alpha-1)}{\|y - x\|} + 1\right) = f(z_2).$$
    We rewrite \eqref{eq:secant} as follows:
    \begin{equation}\label{eq:secant2}
      \dfrac{f(y) - f(z_1)}{r(\alpha-1)} \le \dfrac{f(x)-f(y)}{\|x - y\|} \le \dfrac{f(z_2)-f(x)}{r(\alpha-1)}.  
    \end{equation}
    Furthermore,
    $$\|z_1 - x_0\| \le \left\| y - x_0 \right\|+ r(\alpha-1) \dfrac{\|y - x\|}{\|y-x\|} < r \alpha, $$
    $$\|z_2 - x_0\| \le \left\| x - x_0 \right\|+ r(\alpha-1) \dfrac{\|y - x\|}{\|y-x\|} < r\alpha . $$
    Hence, $z_1,z_2 \in B\left(x_0,\alpha r \right)$. Observe that $r \le  r \alpha\dfrac{1-\lambda}{\lambda}$ and $x,y \in B(x_0,r)$ so $x, y \in B \left(x_0, \alpha r\cdot \dfrac{1-\lambda}{\lambda}\right)$. By \eqref{eq:upper} and \eqref{eq:lower}, we have 
     $$\dfrac{f(z_2) - f(x)}{r(\alpha-1)} \le \dfrac{1}{r(\alpha-1)}\left( M - \dfrac{1}{\lambda}f(x_0) + \dfrac{1 - \lambda}{\lambda}\cdot M\right) = \dfrac{1}{r\lambda(\alpha-1)}(M - f(x_0)), $$
     $$\dfrac{f(y) - f(z_1)}{r(\alpha-1)} \ge \dfrac{1}{r(\alpha-1)}\left( \dfrac{1}{\lambda}f(x_0) - \dfrac{1 - \lambda}{\lambda}\cdot M - M \right) = \dfrac{1}{r\lambda(\alpha-1)} \left( f(x_0) - M \right). $$
     Combined with \eqref{eq:secant2}, we obtain
     $$
      \dfrac{|f(x) - f(y)| }{\|x-y\|} \le \dfrac{1}{r\lambda (\alpha - 1)}\cdot\left[\max \{ f(z) \mid z \in S \} - f(x_0)\right].
     $$
  Since $x$ and $y$ are taken arbitrarily from $B(x_0,r)$, we obtain our conclusion.
\end{proof}

\begin{Remark}\rm Utilizing Lemma \ref{lem:lipschitz} with a thoughtful choice of $S$ produces a Lipschitz constant for the function on a given ball $B(x_0,r)$ at the cost of calculating the function value at a relative small number of points. As an example, for any convex function $f: \R^n \to \R$, with the choice
\begin{equation}\label{eq:S}
    S:= \{ x_0 \pm \alpha nre_i \mid i = 1,2,\ldots,n\}. 
\end{equation}
    where $\{e_i | i = 1,2,\ldots,n\}$ is the standard basis of $\R^n$, Lemma \ref{lem:lipschitz} only requires $2n$ function values to generate a Lipschitz constant. Indeed, to apply Lemma \ref{lem:lipschitz} for the set in \eqref{eq:S}, we need to show that $B(x_0,\alpha r) \subset \conv (S)$. Take any $x \in B(x_0,\alpha r)$, then we can write $x = x_0 + \alpha ru, $
where $\|u\| \le 1$. Next, we write $u = \sum\limits_{i=1}^n \lambda_i e_i$
where for all $i = 1,2,\ldots,n$, $|\lambda_i| \le \sqrt{\sum\limits_{i=1}^n \lambda_i^2} = \|u\| \le 1.$ Hence, for each $i$, we can write $\lambda_i = t_i(-1) + (1-t_i)1$ for $t_i \in [0,1]$. Thus,
$$\begin{aligned}
 x = x_0 + \alpha r \sum\limits_{i=1}^n \lambda_i e_i &= \sum\limits_{i=1}^n \dfrac{1}{n}\left( x_0 + \alpha r n\lambda_i e_i  \right) = \sum\limits_{i=1}^n \dfrac{1}{n}\Big[ x_0 + \alpha rn\big(t_i(-1) + (1-t_i)1\big) e_i  \Big]\\
 &= \sum\limits_{i=1}^n \dfrac{1}{n}\big[ t_i(x_0 - \alpha rne_i) + (1- t_i)(x_0 + \alpha rne_i) \big].    
\end{aligned}$$
Hence, $x \in \conv (S)$, which establishes the needed inclusion. Note that we can further reduce the number of required points from $2n$ to $n+1$ (but not beyond) by finding a simplex that contains the ball $B(x_0,r)$. To see why a $n$-dimensional simplex is optimal for this purpose, a polytope that contains the ball $B(x_0,r)$ must take $\R^n$ as its affine hull, and for this to occur, that polytope must have at least $n+1$ generators.
\end{Remark}

Next, we attempt to apply Lemma \ref{lem:lipschitz} to find a global Lipschitz constant for eligible functions, and consequently establish a characterization for globally Lipschitz continuous convex functions on $\R^n$. To see this end, let us mention an important observation (see \cite[Lemma 5.8]{thesis}): For any $\varepsilon > 0$, there exists a finite set of points $S$ such that $B(0,1) \subset \conv (S) \subset B(0,1+\varepsilon)$. Let $\mathrm{extr}(\conv (S))$ be the set of extreme points of $\conv (S)$. Then, $\conv (S) = \conv (\mathrm{extr}(\conv (S)))$ and $1 \le \|x\| \le 
1 + \varepsilon$ for all $x \in \mathrm{extr}(\conv (S))$. We then proceed to replace $S$ by $\mathrm{extr}(\conv (S))$ (which has a finite number of elements), and via a suitable affine transformation, one easily obtains the following:

\begin{Lemma}\label{lem:polyhedral} For any $\varepsilon > 0$, and any $x_0 \in \R^n$, $r > 0$, there exists a finite set of points $S$ such that $S \subset B(x_0,r+\varepsilon) \setminus B(x_0,r)$ and $B(x_0,r) \subset \conv (S) \subset B(x_0,r+\varepsilon)$.
\end{Lemma}

Combining Lemma \ref{lem:lipschitz} and Lemma \ref{lem:polyhedral} allows us to establish the following result.

\begin{Theorem} \label{theo:lipschitz} Let $f: \R^n \to \R$ be convex. Then $f$ is globally Lipschitz continuous if and only if $\ell:= \limsup\limits_{\|x\| \to +\infty} \dfrac{|f(x)|}{\|x\|}<+\infty$. In this case, $\ell$ is also the global Lipschitz modulus of $f$.
\end{Theorem}
\begin{proof} 
To justify the ``only if" part, suppose that $f$ is globally Lipschitz continuous for some constant $\sigma\in (0,+\infty)$ and show that $\ell<+\infty$. Indeed, for any $x\neq 0$,
\begin{equation*}
    \dfrac{|f(x)|}{\|x\|}\le \dfrac{|f(x)-f(0)|}{\|x-0\|}+\dfrac{|f(0)|}{\|x\|} \le \sigma + \dfrac{|f(0)|}{\|x\|} \to \sigma  \ \text{ as }\ \|x\|\to +\infty,
\end{equation*}
which thus implies that $\ell= \limsup\limits_{\|x\| \to +\infty} \dfrac{|f(x)|}{\|x\|}<+\infty$.

Next, we verify the ``if" part by fixing $x_0\in \R^n$, $\lambda \in (0,1)$ and $\alpha >\max\left\{1,\dfrac{\lambda}{1-\lambda}\right\}$. Applying Lemma \ref{lem:polyhedral} for all $r > 0$, we can find a finite set of points $S_{\alpha,r}$ such that $S_{\alpha,r} \subset B(x_0,\alpha r+1) \setminus B(x_0,\alpha r)$ and $B(x_0,\alpha r) \subset \conv (S_{\alpha,r}) \subset B(x_0,\alpha r + 1)$. Since $B(x_0,\alpha r) \subset \conv(S_{\alpha,r})$ for all $r > 0$, applying Lemma \ref{lem:lipschitz} yields
a Lipschitz constant $L(\lambda,\alpha,r)$ of $f$ on $B(x_0,r)$ given by
\begin{equation}\label{eq:Lipschitzseq}
    L(\lambda,\alpha,r) = \dfrac{1}{r\lambda (\alpha-1)}\cdot \left[ \max \{f(z) \mid z \in S_{\alpha,r} \} - f(x_0) \right].
\end{equation}
In light of \eqref{eq:Lipschitzseq}, for each $r > 0$ pick any $z_{\alpha,r} \in \mathrm{argmax} \{ f(z) \mid z \in S_{\alpha,r} \}$. By the construction of $S_{\alpha,r}$, it follows that $\alpha r \le \|z_{\alpha,r} - x_0\| \le \alpha r + 1$, which implies that $\|z_{\alpha,r}\| \to +\infty$ and $\dfrac{\|z_{\alpha,r}\|}{r} \to \alpha$ as $r\rightarrow +\infty$. Therefore,
$$\begin{aligned}
    \limsup\limits_{r \to +\infty} L(\lambda,\alpha,r) &= \dfrac{1}{\lambda (\alpha-1)} \cdot \limsup\limits_{r \to +\infty} \left( \dfrac{1}{r}\cdot \left[ \max\{ f(z) \mid z \in S_{\alpha,r} \} - f(x_0) \right]\right) = \dfrac{1}{\lambda (\alpha-1)} \cdot \limsup\limits_{r \to +\infty} \dfrac{f(z_{\alpha,r})}{r}\\
    &\le \dfrac{1}{\lambda (\alpha-1)}   \cdot \limsup\limits_{r \to +\infty} \left( \dfrac{\big|f(z_{\alpha,r})\big|}{\|z_{\alpha,r}\|} \cdot \dfrac{\|z_{\alpha,r}\|}{r} \right) \le \dfrac{1}{\lambda (\alpha-1)}  \cdot \left( \limsup\limits_{r \to +\infty}  \dfrac{\big|f(z_{\alpha,r})\big|}{\|z_{\alpha,r}\|}\right) \cdot \alpha \\
    &\le \dfrac{\alpha}{\lambda (\alpha-1)}  \cdot \ell,
\end{aligned} $$
which justifies the asymptotic upper bound of $L(\lambda,\alpha,r)$ as $r \to +\infty$. Fix any $x, y \in \R^n$, then $x, y \in B(x_0,r)$ for all $r$ large. By the definition of $L(\lambda,\alpha,r)$, we obtain, for all such $r$, the estimate
$$|f(x) - f(y)| \le L(\lambda,\alpha,r) \|x - y\|. $$
Passing to the limits as $r \to +\infty$ yields
$$|f(x) - f(y)| \le \limsup\limits_{r \to +\infty} L(\lambda,\alpha,r) \|x - y\| \le \dfrac{\alpha}{\lambda (\alpha-1)}\cdot \ell \|x - y\|, $$
and further, by letting $\alpha\to +\infty$ one gets
\begin{equation*}
    |f(x) - f(y)| \le \dfrac{1}{\lambda}\cdot \ell \|x - y\|.
\end{equation*}
Finally, passing to the limits as $\lambda \uparrow 1$ yields
\begin{equation*}
    |f(x) - f(y)| \le \ell \|x - y\|.
\end{equation*}
Hence, $f$ is globally Lipschitz continuous with constant $ \ell=\limsup\limits_{\|x\| \to +\infty} \dfrac{|f(x)|}{\|x\|}$. \medskip

Now we show that $\ell$ is the global Lipschitz modulus of $f$. Indeed, assume that $L \ge 0$ is a global Lipschitz constant of $f$, we will show that $L \ge \ell$. Since $L$ is a global Lipschitz constant,
$$\forall x,y \in \R^n, \quad |f(x) - f(y)| \le L\|x-y\|. $$
Hence, for all $x \ne 0$,
$$L \ge \dfrac{|f(x) - f(0)|}{\|x \|} \ge \dfrac{|f(x)| - |f(0)|}{\|x\|}. $$
Passing the limit as $\|x\| \to +\infty$, we obtain
$$L \ge \limsup\limits_{\|x\| \to +\infty} \dfrac{|f(x)| - |f(0)|}{\|x\|} =\limsup\limits_{\|x\| \to +\infty} \dfrac{|f(x)|}{\|x\|} = \ell. $$
Hence, $\ell$ is indeed the global Lipschitz modulus of $f$.
\end{proof}

An immediate corollary of Theorem \ref{theo:lipschitz} provides us with an extension of a well-known property of convex functions. To be specific:

\begin{Corollary} If $f: \R^n \to \R$ is convex, then the following are equivalent:
\begin{itemize}
    \item[\rm \textbf{(i)}] $f$ is constant on $\R^n$;
    \item[\rm \textbf{(ii)}] $f$ is bounded above on $\R^n$;
    \item[\rm \textbf{(iii)}] $\limsup\limits_{\|x\| \to +\infty} \dfrac{|f(x)|}{\|x\|} = 0$.
\end{itemize}
\end{Corollary}
\begin{proof} The statement [(i) $\Longrightarrow$ (ii)] is trivially correct. To see [(ii) $\Longrightarrow$ (iii)], assume that $f$ is bounded above on $\R^n$, we will show that $f$ is also bounded below on $\R^n$. Indeed, let $M \in \R$ be an upper bound of $f$ on $\R^n$ and take any $x \in \R^n$. Then, by the convex inequality, 
$$f(x) \ge 2f(0) - f(-x) \ge 2f(0) - M. $$
Hence, $f$ is lower bounded by $2f(0) - M$. Since $f$ is also upper bounded, it follows that $f$ is bounded on $\R^n$, and (iii) trivially holds. The statement [(iii) $\Longrightarrow$ (i)] is a direct consequence of Theorem \ref{theo:lipschitz} when $\ell = 0$.
\end{proof}

Some further comments on Theorem \ref{theo:lipschitz} are collected below.

\begin{Remark}\rm
\label{rem:subgrad}
\quad 
\begin{enumerate}
\item[{\bf (i)}] In the approach we presented in the proof of Theorem \ref{theo:lipschitz}, there is no need to invoke any information on gradients or generalized derivatives. However, there is an alternative approach to prove that if $\limsup\limits_{\|x\| \to +\infty} \dfrac{|f(x)|}{\|x\|} < +\infty$, or merely if $\ell:= \limsup\limits_{\|x\| \to +\infty} \dfrac{f(x)}{\|x\|} < +\infty$ for some convex function $f$, then $f$ is globally Lipschitz continuous with modulus $\ell$, which requires the use of convex subgradients. Indeed, for all $x, y \in \R^n$, there exists a subgradient $y^* \in \R^n$ of $f$ at $y$ (this existence is guaranteed by \cite[Theorem~1.4.2]{urruty}) such that 
\begin{equation}\label{eq:subg}
f(x) \ge f(y) + \la y^*, x - y \ra.
\end{equation}
From \eqref{eq:subg}, dividing both sides by $\|x\|$ and letting $\|x\| \to +\infty$ yields
\begin{equation}\label{L}
+\infty > \limsup\limits_{\|x\| \to +\infty} \dfrac{f(x)}{\|x\|} \ge {\limsup_{\|x\|\to +\infty}\dfrac{\la y^*,x\ra}{\|x\|}=\|y^*\|}.
\end{equation}
Let $\ell:=\limsup\limits_{\|x\| \to +\infty} \dfrac{f(x)}{\|x\|}$. Combining \eqref{L} with \eqref{eq:subg} and the triangle inequality lends us the estimate
$$f(x) - f(y) \ge \la y^*,x-y\ra \ge -\|y^*\|\cdot \|x-y\| \ge -\ell\|x-y\|,$$
and hence $f(x)-f(y)\ge -\ell\|x-y\|$. Switching the roles of $x$ and $y$ shows that 
$$|f(x) - f(y)| \le  \ell\|x-y\|,$$
i.e. $f$ is globally Lipschitz continuous with constant $\ell = \limsup\limits_{\|x\| \to +\infty} \dfrac{f(x)}{\|x\|} \ge 0$. Furthermore, by Theorem \ref{theo:lipschitz}, $f$ is globally Lipschitz continuous with modulus $\limsup\limits_{\|x\| \to +\infty} \dfrac{|f(x)|}{\|x\|} \ge \ell$. Hence, $\ell$ must be the Lipschitz modulus of $f$ and
\begin{equation}\label{f=|f|}
    \limsup\limits_{\|x\| \to +\infty} \dfrac{f(x)}{\|x\|} = \limsup\limits_{\|x\| \to +\infty} \dfrac{|f(x)|}{\|x\|}.
\end{equation}
 Note that we only established that \eqref{f=|f|} holds when either side is finite. However, if the right hand side of \eqref{f=|f|} is $+\infty$, the left hand side is non-negative by part (i), and thus have to be $+\infty$ as well, since the finiteness of the left hand side implies \eqref{f=|f|} which contradicts the infiniteness we assumed for the right hand side. Hence, \eqref{f=|f|} always holds for any convex function $f: \R^n \to \R$.
\item[{\bf (ii)}] Extensions of the proof given in Remark \ref{rem:subgrad}(i) (and thus, of Theorem \ref{theo:lipschitz} as well) to infinite-dimensional Banach spaces are made possible by imposing, in addition, the lower semicontinuity of $f$ which guarantees the nonemptiness of the subdifferentials (see, e.g., \cite[Proposition~2.111~and~Proposition~2.126(iv)]{BS}). 
   \item[\textbf{(iii)}] If $f$ is non-convex, we cannot obtain the same conclusions as in Theorem \ref{theo:lipschitz}. Indeed, consider the univariate function $$f(x) = \left\{\begin{array}{ll}
        \dfrac{1}{|x|}, &  x \ne 0, \\
        0, & x = 0. 
    \end{array} \right.$$
    Then, $\limsup\limits_{|x| \to +\infty} \dfrac{|f(x)|}{|x|} = 0$ but $f$ is not globally Lipschitzian. Indeed, take $x_k = \dfrac{1}{k}$, $y_k = \dfrac{1}{k+1}$ then $|f(x_k) - f(y_k)| = 1$ but $|x_k-y_k| \to 0$ so there cannot exist a global Lipschitz constant. 
    \item[\textbf{(iv)}] The characterization in Theorem \ref{theo:lipschitz} still holds when we replace the Euclidean norm by any norm, since all norms on $\R^n$ are equivalent.
\end{enumerate}
\end{Remark}

Contemplating further, even for non-convex functions, the ``only if" part in Theorem \ref{theo:lipschitz} also hold. The details are provided as follows.

\begin{Remark}\rm When $\limsup\limits_{\|x\| \to +\infty} \dfrac{|f(x)|}{\|x\|} = +\infty$, even without the assumed convexity of $f: \R^n \to \R$, we can show that $\lim\limits_{r \to +\infty} l(r) = +\infty$ where $l(r)$ is the Lipschitz modulus of $f$ on $B(x_0,r)$, $r > 0$. In other words, in no way can we obtain a finite asymptotic upper bound of any sequence of Lipschitz constants of $f$ on $B(x_0,r)$ as $r \to +\infty$. Indeed, fix any $x \in \R^n \setminus \{0\}$, then for all sufficiently large $r > 0$,
$$\dfrac{|f(x) - f(0)|}{\|x\|} \le l(r). $$
Let $r \to +\infty$, we obtain
$$\dfrac{|f(x) - f(0)|}{\|x\|} \le \lim\limits_{r \to +\infty} l(r). $$
Therefore,
$$\begin{aligned}
    +\infty = \limsup\limits_{\|x\| \to +\infty} \dfrac{|f(x)|}{\|x\|} &\le \limsup\limits_{\|x\| \to +\infty} \left( \dfrac{|f(x) - f(0)|}{\|x\|} +  \dfrac{|f(0)|}{\|x\|}\right)\\
    &\le  \limsup\limits_{\|x\| \to +\infty} \left( \lim\limits_{r \to +\infty} l(r) +  \dfrac{|f(0)|}{\|x\|}\right) =  \lim\limits_{r \to +\infty} l(r).
\end{aligned} $$

In other words, $\lim\limits_{r \to +\infty} l(r) = +\infty$, as we need to prove. Note that this observation also acts as an alternative proof for the ``only if" part in Theorem \ref{theo:lipschitz} while extending it to possibly non-convex functions as well. Namely, a function $f: \R^n \to \R$ is not globally Lipschitz continuous if $\ell := \limsup\limits_{\|x\| \to +\infty} \dfrac{|f(x)|}{\|x\|} = +\infty$.
\end{Remark}

To wrap up our note, the following two examples produce classes of globally Lipschitz continuous convex functions that are not norms, as well as their global Lipschitz modulus.

\begin{example}[logistic loss functions]\rm \label{exam:logloss}
As demonstrated via Theorem \ref{theo:lipschitz} and Remark \ref{rem:subgrad}(i), the global Lipschitz continuity of a convex function $f:\R^n \to \R$ amounts to the fulfillment of the asymptotic condition $\limsup_{\|x\|\to +\infty}\dfrac{f(x)}{\|x\|}<+\infty$. Roughly speaking, such convex function $f$ is bounded above by a stretch/shrink of $\|x\|$ at infinity, which seems a bit restrictive. Still, an illuminating example can be found from practical modelling, e.g., from the logistic loss function in support vector machine problems \cite{logloss}: 
\begin{equation*}
    f(x) = \ln \left( 1+ e^{\mathbf{b}^T x}\right), \quad x\in \R^n,
\end{equation*}
where $\mathbf{b}:=(b_1,b_2,\ldots,b_n)\in \R^n$ represents the optimal hyperplane. Indeed, let us check the convexity and the global Lipschitz continuity of $f$ as well as the asymptotic condition $\limsup_{\|x\|\to +\infty}\dfrac{f(x)}{\|x\|}<+\infty$. One easily checks that
\begin{equation}\label{eq:gradf}
\nabla f(x) =  \dfrac{e^{\mathbf{b}^T x}}{1+e^{\mathbf{b}^T x}}\mathbf{b} = \mathbf{b} - \dfrac{1}{1+e^{\mathbf{b}^T x}}\mathbf{b}
\end{equation}
and 
$$
\nabla^2 f(x) =\dfrac{e^{\mathbf{b}^T x}}{\left(1+e^{\mathbf{b}^T x}\right)^2} \cdot \begin{pmatrix}
    b_1^2 & b_1b_2 &\ldots & b_1 b_n \\
    b_2 b_1 & b_2^2 &\ldots & b_2 b_n \\
    \ldots & \ldots &\ldots &\ldots\\
    b_n b_1 &b_n b_2 &\ldots & b_n^2
    \end{pmatrix} = \dfrac{e^{\mathbf{b}^T x}}{\left(1+e^{\mathbf{b}^T x}\right)^2}\cdot \mathbf{b}\mathbf{b}^T.
$$
Hence, for any $u\in \R^n$,
\begin{equation*}
    u^T \nabla f^2 (x) u = \dfrac{e^{\mathbf{b}^T x}}{\left(1+e^{\mathbf{b}^T x}\right)^2}\cdot u^T \mathbf{b}\mathbf{b}^T u = \dfrac{e^{\mathbf{b}^T x}}{\left(1+e^{\mathbf{b}^T x}\right)^2}\cdot (u^T \mathbf{b})^2 \ge 0
\end{equation*}
which guarantees the convexity of $f$.

As $f$ is continuously differentiable on $\R^n$, for arbitrary $x,y\in \R^n$, it follows from the classical mean value theorem for multivariable functions and \eqref{eq:gradf} that
\begin{equation*}
    |f(x)-f(y)|\le \|\nabla f(z)\|\cdot \|x-y\| \le \dfrac{e^{\mathbf{b}^T x}}{1+e^{\mathbf{b}^T x}}\cdot \|\mathbf{b}\| \cdot \|x-y\| \le \|\mathbf{b}\| \cdot \|x-y\|
\end{equation*}
for some $z$ belonging to the line segment connecting $x,y$. Consequently, $f$ is Lipschitz continuous on $\R^n$ with a constant $\|\mathbf{b}\|$. 

To find an upper bound for the limit superior of $\dfrac{f(x)}{\|x\|}$ as $\|x\|$ tends to $+\infty$, first observe for $x\neq 0$ that 
\begin{equation*}
    \dfrac{f(x)}{\|x\|} = \dfrac{\ln \left(1+e^{\mathbf{b}^T x}\right)}{|\mathbf{b}^T x|+1}\cdot \dfrac{|\mathbf{b}^T x|+1}{\|x\|},
\end{equation*}
where
\begin{equation*}
    \dfrac{\ln \left(1+e^{\mathbf{b}^T x}\right)}{|\mathbf{b}^T x|+1} \le \dfrac{\ln \left(1+e^{|\mathbf{b}^T x|}\right)}{|\mathbf{b}^T x|+1} \le \dfrac{\ln \left(e^{|\mathbf{b}^T x|+1}\right)}{|\mathbf{b}^T x|+1} = \dfrac{|\mathbf{b}^T x|+1}{|\mathbf{b}^T x|+1}=1
\end{equation*}
and 
\begin{equation*}
    \dfrac{|\mathbf{b}^T x|+1}{\|x\|} \le \dfrac{\|\mathbf{b}\|\cdot \|x\|}{\|x\|}+\dfrac{1}{\|x\|} = \|\mathbf{b}\| + \dfrac{1}{\|x\|} \to \|\mathbf{b}\| \ \text{ as }\ \|x\|\to +\infty.
\end{equation*}
Therefore,
\begin{equation*}
    \limsup_{\|x\|\to +\infty}\dfrac{f(x)}{\|x\|} \le \|\mathbf{b}\|<+\infty.
\end{equation*}
\end{example}

Another example to portray the diversity of convex functions which satisfies the asymptotic condition $\limsup\limits_{\|x\| \to +\infty} \dfrac{|f(x)|}{\|x\|}<+\infty$ is provided below, as a representative for non-smooth functions.

\begin{Example}[convex piecewise linear functions]\rm Consider $f: \R^n \to \R$ be defined by
$$f(x) = \max \{ l_1(x), l_2(x),\ldots, l_p(x) \}, $$
where $l_1, l_2, \ldots, l_p$ are affine functions. It is shown in \cite[Theorem 2.49]{rockafellar} that every convex, piecewise linear function $f: \R^n \to \R$ (i.e. a function whose domain is split into a finite number of polyhedrals, on each of which it coincides with an affine function) can be presented using this simple form. For $f$, we will verify the convexity, then the global Lipschitz continuity and concurrently find the global Lipschitz modulus. Indeed, for each $i = 1,2,\ldots,p$, let $l_i(x) := \mathbf{b_i}^Tx + \alpha_i$, where $\mathbf{b_i} \in \R^n$ and $\alpha_i \in \R$. First, we see that $f$ has to be convex on $\R^n$, because each $l_i$ is convex on $\R^n$. Denote $C_i := \{x \in \R^n \mid f(x) = l_i(x)\}$ then
$$\forall i = 1,2,\ldots,p, \  \forall x \in C_i, \quad |f(x)| = |l_i(x)| =  \left| \mathbf{b_i}^Tx + \alpha_i  \right| \le \|\mathbf{b_i}\|.\|x\| + |\alpha_i|. $$
Appealing to the fact that $\R^n = \bigcup\limits_{i=1}^p C_i$, we obtain
$$\forall x \in \R^n, \quad |f(x)| \le \max\limits_{i=1,2,\ldots,p} \{ \|\mathbf{b_i}\|.\|x\| + |\alpha_i| \} \le  \left(\max\limits_{i=1,2,\ldots,p} \|\mathbf{b_i}\|\right) \cdot \|x\| + \max\limits_{i=1,2,\ldots,p} |\alpha_i|. $$
Divide both sides by $\|x\|$ and pass to the limits as $\|x\| \to +\infty$, we obtain
\begin{equation}\label{eq:leq}
   \limsup\limits_{\|x\| \to +\infty} \dfrac{|f(x)|}{\|x\|} \le \max\limits_{i=1,2,\ldots,p} \|\mathbf{b_i}\|. 
\end{equation}
On the other hand, fix any index $i = 1,2,\ldots,p$. It is clear that
$$\forall x \in \R^n, \quad \mathbf{b_i}^Tx + \alpha_i = l_i(x) \le f(x) \le |f(x)|. $$
Let us again divide both sides by $\|x\|$ and let $\|x\| \to +\infty$, which lends us the estimate
$$ \limsup\limits_{\|x\| \to +\infty} \dfrac{\mathbf{b_i}^Tx }{\|x\|}\le \limsup\limits_{\|x\| \to +\infty} \dfrac{|f(x)|}{\|x\|}. $$
Noting that for each $i = 1,2,\ldots,p$, the left hand side is precisely $\| \mathbf{b_i} \|$, since this inequality holds for all indexes $i$, it follows suit that
\begin{equation}\label{eq:geq}
  \max\limits_{i=1,2,\ldots,p} \|\mathbf{b_i}\|  \le \limsup\limits_{\|x\| \to +\infty} \dfrac{|f(x)|}{\|x\|}. 
\end{equation}
Combining \eqref{eq:leq} and \eqref{eq:geq} yields $\limsup\limits_{\|x\| \to +\infty} \dfrac{|f(x)|}{\|x\|} =  \max\limits_{i=1,2,\ldots,p} \|\mathbf{b_i}\| < +\infty$. By Theorem \ref{theo:lipschitz}, $f$ is globally Lipschitz continuous with modulus $\ell = \max\limits_{i=1,2,\ldots,p} \|\mathbf{b_i}\|$.
\end{Example}

\end{document}